\documentclass[a4paper,10pt]{amsart}
\usepackage[active]{srcltx}
\usepackage{mathrsfs}
\usepackage[all]{xy}
\usepackage{calrsfs}
\usepackage{times}
\usepackage[scaled=0.92]{helvet}
\usepackage[greek,english]{babel}
\usepackage[iso-8859-7]{inputenc}
\usepackage{graphicx}

\usepackage{amsfonts}
\usepackage{amssymb}
\usepackage[usenames]{color}

\usepackage[usenames]{color}
\newtheorem{theorem}{Theorem}   

\newtheorem{lemma}[theorem]{Lemma}
\newtheorem{proposition}[theorem]{Proposition}

\theoremstyle{definition}

\definecolor{MyDarkBlue}{rgb}{0,0.08,0.60}

\newcommand{\Spe}{{\rm Spec}}

\renewcommand{\O}{{\mathcal{O}}}
\newcommand{\Z}{{\mathbb{Z}}}

\title{A geometric interpretation of  the  Hasse-Arf theorem}

\date{\today}

\author{Aristides Kontogeorgis}
\address{Department of Mathematics, University of Athens\\
Panepistimioupolis, 15784 Athens, Greece
}
\email{kontogar@math.uoa.gr}

\begin{document}
\bibliographystyle{amsplain}

\begin{abstract}
We give an a  geometric interpretation of the Hasse-Arf theorem for function fields using the 
recently proved Oort conjecture. 
\end{abstract} 

\thanks{{\bf keywords:} Automorphisms, Curves, Local fields, Hasse-Arf, Oort Conjecture. 
 }

\maketitle 
\section{Introduction}
The Hasse-Arf theorem (see \cite[IV.3]{SeL}) 
controls the jumps of the ramification filtration for abelian groups.  
Aim of this short note  is to give an intuitive geometric argument,  why
the Hasse-Arf theorem is true. 
Unfortunately our approach does not 
provide a new proof of the Hasse-Arf theorem; It is based on the 
recently proved Oort conjecture for cyclic groups, and the known proof of 
the Oort conjecture is heavily based on the Hasse-Arf theorem. 
However, we have enjoyed  this argument  and  we believe that it reveals the nature
of the Hasse-Arf theorem. We would like to mention that the author has spent a 
lot of time trying to prove the Hasse-Arf theorem with the methods 
of the article \cite{KaranKontSemi} but without success.

After posting a first version on the arXiv, M. Mattignon informed the author that 
essentially  the same   observation was made in his article \cite[section 6 p.16]{MatignonGreen98} jointly written by B. Green.  He also  explained that  the proof of Sen's \cite{Sen69}  theorem by Lubin 
\cite{Lubin95}, implies the truth of the Hasse-Arf theorem and is really a geometric proof
of the Hasse-Arf theorem. We  would like to thank him for his response. 
We  would like to also thank M. Romagny for his comments.

Let $\O=k[[t]]$ be a complete local ring  over an algebraically closed field $k$ of
characteristic $p>0$.
This ring $\O$ is equipped with a valuation $v$ and a local uniformiser $t$. 
Suppose that a finite $p$-group $G$ acts on $\O$. Then a ramification 
filtration is defined by 
\[
 G_i:=\{ \sigma \in G : v(\sigma(t)-t) \geq i+1\}.
\]
The Hasse Arf theorem reduces \cite[IV.3 exer. 3]{SeL} to the following statement 
for the cyclic case:
\begin{theorem}
 Let $G$ be a cyclic $p$-group of order $p^n$. The ramification filtration
 is given by 
 \[
  G_0=G_1=\cdots=G_{j_0}\gneqq G_{j_0+1}=\cdots=G_{j_1}\gneqq G_{j_1+1} 
  =\cdots= G_{j_{n-1}} \gneqq \{1\},
 \]
i.e. the jumps of the ramification filtration  appear at the integers $j_0,\ldots,j_{n-1}$. 
Then 
\[
 j_k=i_0 + i_1 p +i_2 p^2 +\cdots i_k p^k.
\]
\end{theorem}
Notice that since $G$ is assumed to be a $p$-group 
$G_1=G$.
The Harbater-Katz-Gabber compactification theorem asserts that there is a Galois cover
$X\rightarrow \mathbb{P}^1$ ramified only at one point $P$ of $X$ with Galois group 
$G=\mathrm{Gal}(X/\mathbb{P}^1)=G_1$ such that $G_1(P)=G_1$ and the action of 
$G_1(P)$ on the completed local ring $\O_{X,P}$ coincides with the original 
action of $G_1$ on $\O$.

The Oort conjecture (now a theorem proved by  A.Obus, S. Wewers  \cite{ObusWewers} and F. Pop \cite{PopOort}) states:
\begin{theorem}
Let $X$ be a projective nonsingular curve, defined over an 
algebraically closed field $k$ of characteristic $p>0$, acted on by a cyclic group $G$.
 There is a proper smooth family of curves $\mathcal{X} \rightarrow \Spe A$,
 where $A$ is a local ring with maximal ideal $m$ so that $A/m=k$ and 
 the special fibre of $\mathcal{X}_m=\mathcal{X} \times_{\Spe A} \Spe k$ is the 
 original curve $X$ and the generic fibre $\mathcal{X}_0$ is a curve defined over 
 a field of characteristic $0$. 
\end{theorem}

We will apply the Oort conjecture on the Harbater-Katz-Gabber compactification $X$ coming from 
the local action of $G=\Z/p^n\Z$ on $\O$. We construct a relative notion of horizontal 
ramification divisor. This divisor intersects the special fibre on a single point $P$ and 
the generic fibre at a set with $P_1,\ldots,P_s$ points. 
The Hasse-Arf has the following 
geometric interpretation:
\begin{theorem}
 The numbers $i_k$ are numbers of different  orbits of the action of the group $G$ on the 
 ramified
 points in the generic fibres with size $p^k$. 
\end{theorem}

The method of proof reflects the  philosophy of lifting geometrical  objects   
from positive characteristic to characteristic zero, and use the easier 
characteristic zero case. This was one of the main ideas of modular 
representation theory and is also one of the main ideas in defining cohomology theories 
over Witt rings, see for example \cite{SerreMexico}.

\section{Horizontal Ramification Divisors}
Let $P$ be the unique ramification point on the special fibre. 
Let $\sigma \in G_1(P)$, $\sigma \neq 1$, and let $\tilde{\sigma}$ be a lift 
of $\sigma$ in $\mathcal{X}$. The scheme $\mathcal{X}$ is regular at $P$,
 and the completion of $\mathcal{O}_{\mathcal{X},P}$ is isomorphic to the ring $R[[T]]$.
 Weierstrass preparation theorem \cite[prop. VII.6]{BourbakiComm}   implies that:
\[
\tilde{\sigma}(T)-T=g_{\tilde{\sigma}}(T) u_{\tilde{\sigma}}(T),
\]
where $g_{\tilde{\sigma}}(T)$ is a distinguished Weierstrass polynomial 
of degree $m+1$ and $u_{\tilde{\sigma}}(T)$ is a unit in $R[[T]]$.

The polynomial $g_{\tilde{\sigma}}(T)$ gives rise to a horizontal divisor 
that corresponds to the fixed points of $\tilde{\sigma}$. This 
horizontal divisor might not be reducible. 
The branch divisor corresponds to the union of the fixed points of 
any $\sigma \in G_1(P)$. 
Next lemma gives an alternative definition of a horizontal branch divisor for the 
relative curves $\mathcal{X} \rightarrow \mathcal{X}^G$, that works even  when 
$G$ is not a cyclic group.

\begin{lemma} \label{lemmaBRANCH}
Let $\mathcal{X} \rightarrow \Spe A$ be an $A$-curve, admitting a 
fibrewise action of the finite group $G$, where $A$ is a 
Noetherian local ring. 
Let $S=\Spe A$, and $\Omega_{\mathcal{X}/S}$, $\Omega_{\mathcal{Y}/S}$ be
the sheaves of relative differentials of $\mathcal{X}$ over $S$ and 
$\mathcal{Y}$ over $S$, respectively. Let $\pi:\mathcal{X} \rightarrow 
\mathcal{Y}$ be the quotient map.
The sheaf 
\[
\mathcal{L}(-D_{\mathcal{X}/\mathcal{Y}})= \Omega_{\mathcal{X}/S} ^{-1}
\otimes_S \pi^* \Omega_{\mathcal{Y}/S}. 
\]
is the ideal sheaf the horizontal Cartier divisor 
 $D_{\mathcal{X}/\mathcal{Y}}$. The intersection of $D_{\mathcal{X}/\mathcal{Y}}$ with the special and generic fibre 
of $\mathcal{X}$ gives the ordinary branch divisors for curves.
\end{lemma}
\begin{proof}
We will first prove that 
 the above defined divisor $D_{\mathcal{X}/\mathcal{Y}}$ is indeed 
an effective Cartier divisor. According to \cite[Cor. 1.1.5.2]{KaMa}
it is enough to prove that 
\begin{itemize}
\item  $D_{\mathcal{X}/\mathcal{Y}}$ is a closed subscheme which is flat over $S$. 
\item for all geometric points $\Spe k \rightarrow S$ of $S$, the 
closed subscheme $D_{\mathcal{X}/\mathcal{Y}}\otimes_S k$ of $\mathcal{X} \otimes_S k$ is a
Cartier divisor in $\mathcal{X} \otimes _S k/k$. 
\end{itemize}

In our case the special fibre is a nonsingular curve. 
Since the base is a local ring and the special fibre is nonsingular, 
the deformation $\mathcal{X} \rightarrow \Spe A$ is smooth.  
(See the remark after the definition 3.35 p.142 in \cite{LiuBook}).
The smoothness of the curves $\mathcal{X}\rightarrow S$, 
and $\mathcal{Y}\rightarrow S$, implies that the sheaves 
$\Omega_{\mathcal{X}/S}$ and $\Omega_{\mathcal{X}/S}$ are $S$-flat,
\cite[cor. 2.6 p.222]{LiuBook}. 

On the other hand the sheaf $\Omega_{\mathcal{Y},\Spe A}$ is 
by \cite[Prop. 1.1.5.1]{KaMa}  $\O_{\mathcal{Y}}$-flat. 
Therefore, $\pi^*(\Omega_{\mathcal{Y}, \Spe A})$ is $\O_{\mathcal{X}}$-flat
and  $\Spe A$-flat \cite[Prop. 9.2]{Hartshorne:77}.
Finally, observe  that  the intersection with the special and generic 
fibre is the ordinary branch divisor for curves according to 
\cite[IV p.301]{Hartshorne:77}.
\end{proof}


For a curve $X$ and a branch point $P$ of $X$ we will 
denote by  $i_{G,P}$  the order function of the filtration of $G$ at $P$. 
The Artin  representation of the group $G$ is defined 
by $\mathrm{ar}_P(\sigma)=-f_P i_{G,P}(\sigma)$ for $\sigma\neq 1$ and 
$\mathrm{ar}_P(1)= f_P\sum_{\sigma\neq 1} i_{G,P}(\sigma)$ \cite[VI.2]{SeL}.
We are going to use the Artin representation at both the special 
and generic fibre. In the special fibre we always have $f_P=1$ since the
field $k$ is algebraically closed. The field of quotients of $A$ should 
not be algebraically closed therefore a fixed point there might have $f_P \geq 1$.
The integer  $i_{G,P}(\sigma)$ 
is equal to the multiplicity of $P\times P$ in the intersection of 
$\Delta .\Gamma_\sigma$ in the relative $A$-surface 
$\mathcal{X} \times_{\Spe A} \mathcal{X}$, 
where $\Delta$ is the 
diagonal and $\Gamma_\sigma$ is the graph of $\sigma$ \cite[p. 105]{SeL}. 
  
Since the diagonals  $\Delta_0,\Delta_\eta$  and the graphs of $\sigma$ in the special and generic fibres respectively of 
$\mathcal{X}\times_{\Spe A} \mathcal{X}$  are  algebraically equivalent divisors   we have: 
\begin{proposition}\label{bertin-gen}
	Assume that $A$ is an integral domain, and let $\mathcal{X}\rightarrow \Spe A$
	be a deformation of $X$. 
        Let $\bar{P}_i$, $i=1,\cdots,s$ be the  horizontal branch divisors 
        that intersect at the special fibre, at point $P$, and let $P_{i}$ be 
        the corresponding points on the generic fibre. For the Artin 
        representations attached to the points $P,P_{i}$ we have:
        \begin{equation} \label{equali}
        \mathrm{ar}_P(\sigma)=\sum_{i=1}^s \mathrm{ar}_{P_{i}}(\sigma). 
        \end{equation}
\end{proposition}
This generalizes a result of  J. Bertin \cite{BertinCRAS}. Moreover 
if we set $\sigma=1$ to the above formula we obtain a relation 
for the valuations of the differents 
 in the special and the generic fibre, since the value 
of the Artin's representation at $1$ is the valuation of 
the different \cite[prop. 4.IV,prop. 4.VI]{SeL}.  This observetion 
is equivalent to claim  3.2 in \cite{MatignonGreen98} and is one
 direction of a local 
criterion for good reduction theorem proved in \cite[3.4]{MatignonGreen98},
\cite[sec. 5]{KatoDuke87}.

\subsection{The Artin representation on the generic fibre}
We can assume that after a base change of the family $\mathcal{X} \rightarrow \Spe(A)$
the points $P_i$ at the generic fibre have degree $1$. 
Observe also that 
at the generic fibre the Artin representation can be computed as follows: 
\[
\mathrm{ar}_{Q}(\sigma)=\left\{
\begin{array}{l}
1 \mbox{ if } \sigma(Q)=Q,\\
0 \mbox{ if } \sigma(Q)\neq Q. 
\end{array}
\right.
                                  \]
The set of points $S:=\{P_1,\ldots,P_s\}$ that are the intersections of the ramification 
divisor and the generic fibre are acted on by the group $G$. 
Let $S_k$ be the subset of $S$ fixed by $\Z/p^{n-k}\Z$, i.e.
\[
 P\in S_k \mbox{ if and only if } G(P)=\Z/p^{n-k}\Z.
\]
Let $s_k$ be the order of $S_k$. 
Observe that since for a point $Q$ in the generic fibre  $\sigma(Q)$ and $Q$ have
the same stabilizers (they are conjugate but $G$ is abelian) the sets $S_k$ are 
acted on by $G$. Therefore orders of $S_k$ are $s_k=p^{k}i_k$ where  $i_k$ 
is the number of orbits 
of the action of $G$ on $S_k$.
                                  
Observe that 
\[
 G_{j_k}=\left\{
 \begin{array}{ll}
\Z/p^{n-k}\Z & \mbox{ for } 0 \leq k \leq n-1 \\
 \{1\} & \mbox{ for } k \geq n.
 \end{array}
\right.
\]
An element in $G_{j_k}$ fixes only elements with stabilizers that contain $G_{j_k}$. 
So $G_{j_0}$ fixes only $S_0$, $G_{j_1}$ fixes both $S_0$ and $S_1$ and 
$G_{j_k}$ fixes all elements in $S_0,S_1,\ldots,S_k$. 
So eq. (\ref{equali}) implies that an element $\sigma$ in $G_{j_k}-G_{j_{k+1}}$ 
satisfies $\mathrm{ar}_P(\sigma)=j_k$ and by using equation (\ref{equali}) we arive at 
\[
 j_k=i_0+ p i_1 + \cdots p^k i_k.
\]
This completes the proof of the Hasse-Arf theorem. The argument is illustrated in figure
\ref{picture}.
\begin{figure}[h]
\caption{The horizontal Ramification divisor \label{picture}}
\begin{center}
 \includegraphics[scale=0.6]{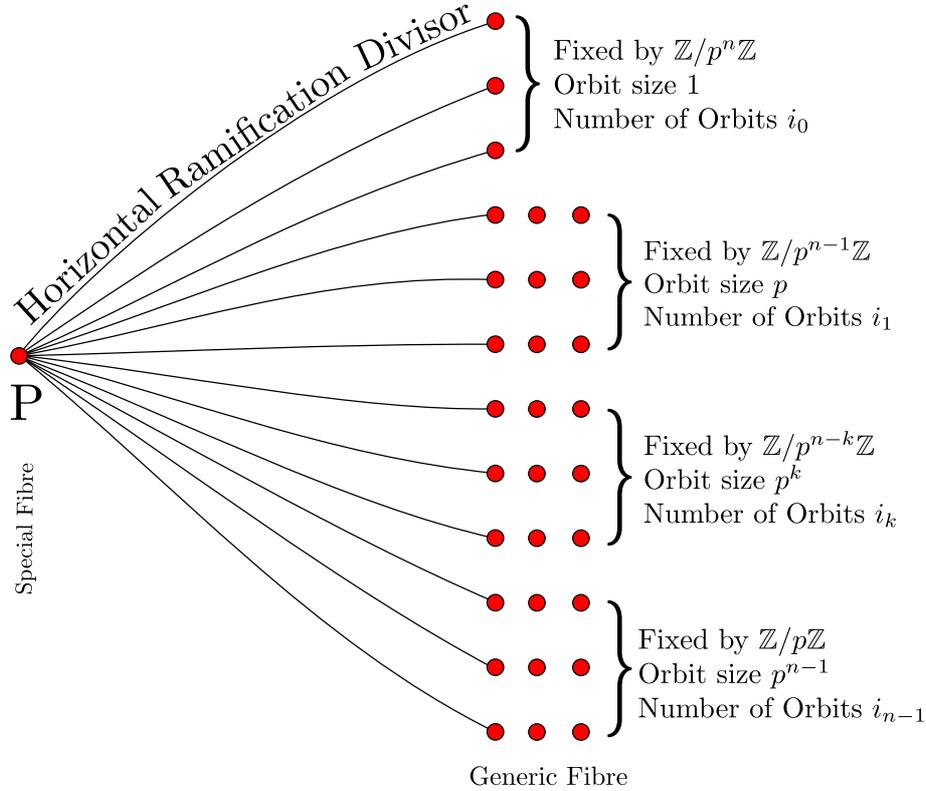}
\end{center}
\end{figure}
                                  
%
%
%
%

 \def\cprime{$'$}
\providecommand{\bysame}{\leavevmode\hbox to3em{\hrulefill}\thinspace}
\providecommand{\MR}{\relax\ifhmode\unskip\space\fi MR }
\providecommand{\MRhref}[2]{%
  \href{http://www.ams.org/mathscinet-getitem?mr=#1}{#2}
}
\providecommand{\href}[2]{#2}

\end{document}